\documentclass[12pt]{amsart}
\allowdisplaybreaks[1]

\usepackage{amssymb,amsmath,amsfonts,amsthm,nomencl,mathrsfs}
\usepackage[latin1]{inputenc}
\newcommand{\IC}{\mathbb{C}}
\newcommand{\IR}{\mathbb{R}}

\newcommand{\z}{ \Id \mu }

\newcommand{\IN}{\mathbb{N}}

\newcommand{\Id}{{\rm d}}

\newcommand{\f}{\frac}
\newcommand{\nn}{\nonumber}

\newtheorem{theorem}{THEOREM}[section]

\newtheorem{Remark}[theorem]{Remark}

\newtheorem{Theorem}[theorem]{Theorem}

\newtheorem{Proposition}[theorem]{Proposition}

\newtheorem{Definition}[theorem]{Definition}

\begin{document}

\title{Sequences of Laplacian cut-off functions}

\author[B. G\"uneysu]{Batu G\"uneysu}
\address{Batu G\"uneysu, Institut f\"ur Mathematik, Humboldt-Universit\"at zu Berlin, 12489 Berlin, Germany} \email{gueneysu@math.hu-berlin.de}

\maketitle 

\begin{abstract} 
We derive several new applications of the concept of \emph{sequences of Laplacian cut-off functions} on Riemannian manifolds (which we prove to exist on geodesically complete Riemannian manifolds with nonnegative Ricci curvature): In particular, we prove that this existence implies $\mathsf{L}^q$-estimates of the gradient, a new density result of smooth compactly supported functions in Sobolev spaces on the whole $\mathsf{L}^q$-scale, and a slightly weaker and slightly stronger variant of the conjecture of Braverman, Milatovic and Shubin on the nonnegativity of $\mathsf{L}^2$-solutions $f$ of $(-\Delta+1)f\geq 0$. The latter fact is proved within a new notion of positivity preservation for Riemannian manifolds which is related to stochastic completeness.
\end{abstract}

\section{Setting and notation}

Let $M\equiv (M,g)$ be a connected smooth Riemannian $m$-manifold without boundary. We will denote the corresponding negative Laplace-Beltrami operator with $\Delta$, the distance function with $\Id(\bullet,\bullet)$, the open balls with $\mathrm{B}_a(x)$, $x\in M$, $a>0$, and the volume measure with $\mu(\Id x):=\mathrm{vol}(\Id x)$, where we shall often simply write $\int f \z$ instead of $\int_M f\z$. Let $p(\bullet,\bullet,\bullet)$ denote the minimal positive heat kernel of $M$, and for any $x\in M$ let $\mathbb{P}^x$ denote the law of a Brownian motion on $M$ starting from $x$, with $\zeta$ the lifetime of continuous paths with explosion on $M$. Then the Kato class $\mathsf{K}(M)\supset \mathsf{L}^{\infty}(M)$ is defined to be the complex linear space of Borel functions $v:M\to\IC$ such that 
$$
\lim_{t\to 0+}\sup _{x\in M}\int^t_0\mathbb{E}^x\left[1_{\{s<\zeta\}}|v(X_s)|\right] \Id s=0,
$$
or equivalently\footnote{We use here the more common normalization that $\mathbb{P}^x$ is a $-\Delta/2$ diffusion (and not a $-\Delta$ diffusion).}
$$
\lim_{t\to 0+}\sup_{x\in M}\int^t_0\int_Mp(s,x,y)|v(y)|\mu(\Id y) \Id s=0.
$$
If $E\to M$ is a smooth Hermitian vector bundle, then whenever there is no danger of confusion, we will denote the underlying Hermitian structure simply with $(\bullet,\bullet)_x$, $x\in M$, $\left|\bullet\right|_x:=\sqrt{(\bullet,\bullet)}_x$ will stand for the corresponding norm and also for the operator norm on $E_x$. Using $\mu$ we get the corresponding $\mathsf{L}^q$-spaces of sections $\Gamma_{\mathsf{L}^q}(M,E)$, whose norms and operator norms are denoted with $\left\|\bullet\right\|_q$, $q\in [1,\infty]$, with $\left\langle\bullet,\bullet \right\rangle$ the canonical scalar product on $\Gamma_{\mathsf{L}^2}(M,E)$. The symbol '$\dagger$' will denote the formal adjoint with respect to $\left\langle\bullet,\bullet \right\rangle$ of a smooth linear partial differential operator that acts on some $\Gamma_{\mathsf{C}^{\infty}_{\mathrm{c}}}(M,E)$. For example, we have $-\Delta=\Id^{\dagger} \Id$. \\
We equip all smooth tensor type bundles corresponding to $\mathrm{T} M$ with their canonical smooth Euclidean structure, where we will freely identify $\mathrm{T}^* M$ with $\mathrm{T} M$ with respect to $g$. For example, the gradient of some $f\in\mathsf{C}^k(M)$ simply becomes $\Id f\in\Omega^1_{\mathsf{C}^{k-1}}(M)$,  where we write $\Omega^1_{\#}(M)=\Gamma_{\#}(M,\mathrm{T}^* M)$ for spaces of $1$-forms. We let $\nabla$ denote the Levi-Civita connection on $M$ and complexify these data in the sequel. 

\section{Main results}

\subsection{Cut-off functions}

The aim of this paper is to derive several important applications of the following concept:

\begin{Definition}\label{def} We say that $M$ admits a \emph{sequence $(\chi_n)\subset \mathsf{C}^{\infty}_{\mathrm{c}}(M)$ of Laplacian cut-off functions}, if $(\chi_n)$ has the following properties\emph{:}
\begin{itemize}
\item[\emph{(C1)}] $0 \le \chi_n(x) \le 1$ for all $n\in\IN$, $x \in M$,

\item[\emph{(C2)}] for all compact $K\subset M$, there is an $n_0(K)\in\IN$ such that for all $n\geq n_0(K)$ one has $\chi_n\mid_{K}= 1$,

\item[\emph{(C3)}] $\left\|\Id\chi_n \right\|_{\infty}:=\sup_{x \in M} \left|\Id\chi_n(x)\right|_x \to 0$ as $n\to \infty$,

 \item[\emph{(C4)}] $\left\|\Delta\chi_n\right\|_{\infty}:=\sup_{x \in M} \left|\Delta \chi_n(x)\right| \to 0$ as $n\to\infty$. 
\end{itemize}
 \end{Definition}

It is intuitively clear that geodesic completeness cannot be dropped in this context. Ultimately, one has:

\begin{Theorem}\label{main}\emph{a)} $M$ is geodesically complete, if and only if $M$ admits a \emph{sequence $(\chi_n)\subset \mathsf{C}^{\infty}_{\mathrm{c}}(M)$ of first order cut-off functions}, meaning that $(\chi_n)$ satisfies (C1), (C2), (C3).\\
\emph{b)} Assume that $M$ is geodesically complete with a nonnegative Ricci curvature. Then $M$ admits a sequence of Laplacian cut-off functions. More precisely, for any sequence $(a_n) \subset (0,\infty)$ with $a_n\to \infty$ there is a sequence $(\chi_n)\subset \mathsf{C}^{\infty}_{\mathrm{c}}(M)$ with (C1), (C2) and 
$$
\left\|\Id\chi_n\right\|_{\infty} =\mathrm{O}(1/a_n), \> \left\|\Delta \chi_n\right\|_{\infty}=\mathrm{O}(1/a_n^2), \>\>\>\text{  $n\to\infty$.}
$$
\end{Theorem}

\begin{proof} a) See Proposition 4.1 in \cite{shubin} for the \lq\lq{}only if\rq\rq{} part. The other direction should be well-known (cf. \cite{pigo} and the references therein where this problem is considered in the context of $p$-parabolicity). We give the short proof for the convenience of the reader: Assume that $M$ admits a sequence of first order cut-off functions $(\chi_n)$. Then given $\mathscr{O}\in M$, $r>0$, we are going to show that there is a compact set $K(\mathscr{O},r)\subset M$ such that $\Id(x,\mathscr{O})>r$ for all $x\in M\setminus K(\mathscr{O},r)$, which implies that any open geodesic ball is relatively compact. To see this, we pick a compact $K(\mathscr{O})\subset M$ such that $\mathscr{O}\in K(\mathscr{O})$, and a number $n(\mathscr{O},r)\in \IN$ large enough such that $\chi_{n(\mathscr{O},r)} =1$ on $K(\mathscr{O})$ and 
$$
\sup_{x \in M} \left|\Id\chi_{n(\mathscr{O},r)}(x)\right|_x\leq 1/(r+1).
$$
Then with $K(\mathscr{O},r):=\mathrm{supp}(\chi_{n(\mathscr{O},r)})$ one easily gets  
$$
\Id(x,\mathscr{O})\geq r+1 \> \text{ for all $x\in M\setminus K(\mathscr{O},r)$.}
$$
b) We will combine a highly nontrivial result from Riemannian rigidity theory, with a careful scaling argument. The following result has been extablished by Cheeger/Colding and Wang/Zhu (cf. Lemma 1.4 in \cite{wang}, the proof of which is based on arguments from \cite{cheeger}): There is a constant $C(m)>0$, which only depends on $m$, such that for any $\mathscr{O}\in M$ and any geodesically complete Riemannian structure $\tilde{g}$ on $M$ with $\mathrm{Ric}_{\tilde{g}}\geq 0$, there is a function $\chi_{\tilde{g}}=\chi_{\tilde{g},\mathscr{O}}\in \mathsf{C}^{\infty}(M)$ which satisfies
\begin{align}\label{wnn0}
&0\leq \chi_{\tilde{g}}\leq 1,\>\mathrm{supp}(\chi_{\tilde{g}})\subset \mathrm{B}_{\tilde{g},2}(\mathscr{O}), \>\chi_{\tilde{g}}=1\text{ on }\mathrm{B}_{\tilde{g},1}(\mathscr{O}), \\
&|\Id \chi_{\tilde{g}}|_{\tilde{g}}\leq C(m),\> |\Delta_{\tilde{g}}\chi_{\tilde{g}}|\leq C(m).\label{wnn}
\end{align}
Now fix some $\mathscr{O}\in M$. Then for $g_n:=\f{4}{a_n^2}g$ we have
\begin{align}
&\mathrm{B}_{g_n,\f{2a}{a_n}}(\mathscr{O})=\mathrm{B}_a(\mathscr{O})\>\text{ for any $a>0$},\nn\\
&\mathrm{Ric}_{g_n}=\mathrm{Ric}\geq 0, \>\>\Delta_{g_n}=\f{a_n^2}{4}\Delta,\nn\\
&|\alpha|^2_{g_n}=\f{a_n^2}{4}|\alpha|^2\>\text{ for any $\alpha\in\Omega^1_{\mathsf{C}^{\infty}}(M)$,}\nn
\end{align}  
thus $\mathrm{B}_{g_n,2}(\mathscr{O})=\mathrm{B}_{a_n}(\mathscr{O})$, $\mathrm{B}_{g_n,1}(\mathscr{O})=\mathrm{B}_{\f{a_n}{2}}(\mathscr{O})$, and the claim follows from setting $\chi_{n}:=\chi_{g_n}$, where $\chi_{g_n}$ is chosen with (\ref{wnn0}), (\ref{wnn}) for $\tilde{g}=g_n$. 
\end{proof}

Let us continue with some remarks on Theorem \ref{main}. Firstly, the existence of a sequence of Laplacian cut-off functions has only been established so far under the rather restrictive \lq\lq{}$\mathsf{C}^{\infty}$-bounded geometry\rq\rq{} assumption on $M$ (cf. Proposition B.3 in \cite{Br}), meaning that $M$ has a positive injectivity radius and all Levi-Civita derivatives of the curvature tensor of $M$ have to be bounded, a result which is thus considerably improved by Theorem \ref{main} b) (in the class of $M$\rq{}s with nonnegative Ricci curvature).\\
Secondly, we emphasize that a lower bound on the Ricci curvature is enough in Theorem \ref{main} b) to guarantee (C4) is rather suprising in the following sense: A canonical \emph{explicit} approach (see e.g. the proof of Proposition 4.1 in \cite{shubin}) for constructing a sequence of cut-off functions is clearly given by mollifying functions of the type $f(\Id(\bullet,\mathscr{O})/n)$, with $n\in\IN$, $f:\IR\to [0,1]$ appropriate, and some fixed $\mathscr{O}\in M$. Here, one is ultimately confronted with estimating $|\Delta \Id(\bullet,\mathscr{O}) |$ away from the cut locus $\mathrm{Cut}(\mathscr{O})\cup \{\mathscr{O}\}$. Thus, comparison theorems suggest the necessity of both, a lower $\emph{and}$ an upper bound on the curvature to get (C4). In this sense, the fact that Theorem \ref{main} only requires a lower bound on the Ricci curvature to get (C4) is surprising. Indeed, our proof relied on the highly nontrivial existence of the $\chi_{\tilde{g}}$\rq{}s, which is proved in \cite{wang,cheeger} \emph{implicitely} by solving a properly chosen Dirichlet problem on the annulus $\mathrm{B}_{\tilde{g},2}(\mathscr{O})\setminus \overline{\mathrm{B}_{\tilde{g},1}(\mathscr{O})}$. As a consequence, one then only has to use the Laplacian comparison theorem in combination with the maximum principle to control $\chi_{\tilde{g}}$ and $\Id \chi_{\tilde{g}}$, and not to control $\Delta_{\tilde{g}}\chi_{\tilde{g}}$.  \vspace{1mm}

The rest of this paper is devoted to several applications of the concept of sequences of Laplacian cut-off functions (and thus of Theorem \ref{main} b)), aiming to illustrate the usefulness of this highly global concept: In Section \ref{graff}, we will prove that the existence of such a sequence together with a lower bound on the Ricci curvature implies Gagliardo-Nirenberg type $\mathsf{L}^q$-estimates of the gradient. Then in Section \ref{sed}, we are going to prove that $\mathsf{C}^{\infty}_{\mathrm{c}}(M)$ is dense in $\mathsf{H}^{2,q}(M)$, if $M$ admits a sequence of Laplacian cut-off functions and carries a $\mathsf{L}^q$-Calderon-Zygmund inequality. Finally, in Section \ref{bams}, we are going to introduce a new concept of 
$$
\text{\lq\lq{}positivity preservation with respect to a class $\mathscr{C}\subset \mathsf{L}^1_{\mathrm{loc}}(M)$\rq\rq{}},
$$
which for $\mathscr{C}=\mathsf{L}^2(M)$ is related to a conjecture of Braverman, Milatovic, Shubin, and for $\mathscr{C}=\mathsf{L}^{\infty}(M)$ to stochastic completeness. Here, we will prove that if $M$ admits a sequence of Laplacian cut-off functions and has a Kato decomposable Ricci curvature, then $M$ is $\mathsf{L}^q(M)$-positivity preserving for all $q\in [1,\infty]$. The proof of this fact relies on a new Markoff type result for covariant Schrödinger operators, which should be of independent interest.

\subsection{$\mathsf{L}^{q}$-properties of the gradient}\label{graff}

Let $D(m):=(2+\sqrt{m})^2$. We start by proving the following $\mathsf{L}^{q}$-properties of the gradient: 

\begin{Theorem}\label{lp} Let
$$
\mathsf{F}(M):=\left.\big\{  \Psi\right| \Psi\in  \mathsf{C}^{2}(M)\cap \mathsf{L}^{\infty}(M)\cap \mathsf{L}^{2}(M), \>\Delta\Psi \in\mathsf{L}^2(M) \big \}.
$$ 
If $M$ admits a sequence of Laplacian cut-off functions and satisfies $\mathrm{Ric}\geq -C$ for some constant $C>0$, then one has 
$$
|\Id \Psi| \in \bigcap_{q\in [2,4]}\mathsf{L}^q(M)\>\text{ for any $\Psi\in \mathsf{F}(M)$}.
$$
More precisely, in this situation, for all $\Psi\in \mathsf{F}(M)$ one has 
\begin{align*}
\left\|\Id \Psi\right\|^2_2= \left\langle \Psi,(-\Delta)\Psi\right\rangle, \left\|\Id \Psi\right\|^4_4 \leq D(m) \left\|\Psi \right\|^2_{\infty}\left(\left\|\Delta\Psi\right\|^2_2+C  \left\|\Id \Psi\right\|^2_2\right).
\end{align*}
\end{Theorem}

Note that the assumptions of Theorem \ref{lp} on $M$ are satisfied, if $M$ is geodesically complete with a nonnegative Ricci curvature. We will need the following generally valid Gagliardo-Nirenberg type inequality for the proof of Theorem \ref{lp}, which should be of an independent interest:

\begin{Proposition}\label{gag} For all $\Psi\in\mathsf{C}^{2}_{\mathrm{c}}(M)$ one has
\begin{align}
\int\left|\Id \Psi \right|^4\z \leq D(m) \left\|\Psi \right\|^2_{\infty}\left(\int |\Delta\Psi|^2\z-\int \mathrm{Ric}(\Id\Psi,\Id \Psi)\z\right).\label{dfdx}
\end{align}
\end{Proposition}

The proof of Proposition \ref{gag} relies on Bochner\rq{}s identity and is implicitely included in the proof of Lemma 2 from \cite {grummt}, where, being motivated by the Eudlidean situation \cite{Lein}, this result is used in the context of essential self-adjointness problems corresponding to Schrödinger operators with singular magnetic potentials. The classical references for (\ref{dfdx}) are \cite{gag, niren}.

\begin{proof}[Proof of Theorem \ref{lp}] We can assume that $\Psi$ is real-valued. Firstly, combining Corollary 2.5 and Corollary 2.6 from \cite{strich} implies that $|\Id\Psi|$ is in $\mathsf{L}^2(M)$ and that the global integration by parts identity 
\begin{align}
\left\|\Id \Psi\right\|^2_2= \left\langle \Psi,(-\Delta)\Psi\right\rangle<\infty\label{l2}
\end{align}
holds true, a fact which only requires geodesic completeness and that $\Psi$, $\Delta\Psi$ are in $\mathsf{L}^2(M)$. Thus, it is sufficient to prove the asserted estimate for $\left\|\Id \Psi \right\|_4^4$. To this end, let $(\chi_n)$ be a sequence of Laplacian cut-off functions, and note that Proposition \ref{gag} implies the following inequality for each $n$:
\begin{align}
&\int\left|\Id (\chi_n\Psi) \right|^{4}\z \leq D(m) \left\|\chi_n\Psi \right\|^{2}_{\infty}\int\left[ \big(\Delta(\chi_n\Psi)\big)^2+C  |\Id (\chi_n\Psi)|^2\right]\z.\nn
\end{align}
Let us note that
\begin{align}
 \Id (\chi_n\Psi)=  \Psi\Id\chi_n+\chi_n\Id\Psi\to \Id \Psi\>\text{ pointwise as $n\to\infty$.}\label{erst}
\end{align}
In view of (\ref{erst}), the latter inequality in combination with Fatou\rq{}s lemma and (C1) shows that the proof is complete, if we can show 
\begin{align}
&\lim_{n\to \infty}\int \big(\Delta(\chi_n\Psi)\big)^2\z=\int (\Delta\Psi)^2\z,\label{ersti}\\
&\lim_{n\to \infty}\int|\Id (\chi_n\Psi)|^2\z=\int|\Id\Psi|^2\z\label{zweit}.
\end{align}
The limit relation (\ref{ersti}) follows from
\begin{align}
\Delta(\chi_n\Psi)=\Psi\Delta\chi_n+2(\Id\chi_n,\Id\Psi)+\chi_n\Delta\Psi, \label{frf}
\end{align}
so that
\begin{align*}
\big(\Delta(\chi_n\Psi)\big)^2=&\Psi^2(\Delta\chi_n)^2
+4\Psi(\Delta\chi_n)(\Id\chi_n,\Id\Psi)
+2\Psi(\Delta\chi_n)\chi_n\Delta\Psi\\
&+4(\Id\chi_n,\Id\Psi)^2
+4(\Id\chi_n,\Id\Psi)\chi_n\Delta\Psi+\chi_n^2(\Delta\Psi)^2, 
\end{align*}
using (C4) for the first, (C2), (C3) and (\ref{l2}) for the second, (C1), (C4) and dominated convergence for the third, (C3) and (\ref{l2}) for the fourth, (C1), (C3), (\ref{l2}) and dominated convergence for the fifth, and finally (C1), (C2) and dominated convergence for the sixth term. \\
Finally, (\ref{zweit}) follows from (\ref{erst}) and
$$
|\Id (\chi_n\Psi)|^2\leq 2|\Psi|^2|\Id\chi_n|^2+2|\chi_n|^2|\Id\Psi|^2
$$
and dominated convergence, using (C1), (C2), (C3). This completes the proof. 
\end{proof}

\subsection{Denseness of $\mathsf{C}^{\infty}_0(M)$ in $\mathsf{H}^{2,q}(M)$}\label{sed}

This section is devoted to a denseness result for the Sobolev space $\mathsf{H}^{2,q}(M)\subset \mathsf{L}^{q}(M)$, $q\in (1,\infty)$, which can be defined as the Banach space given by the closure of 
\begin{align*}
&\left.\big\{  \Psi\right| \Psi\in  \mathsf{C}^{\infty}(M)\cap \mathsf{L}^{q}(M), \>|\Id\Psi|,|\nabla\Id \Psi| \in\mathsf{L}^q(M) \big \}\>\text{ with respect to}\\
&\left\| \Psi\right\|_{2,q}:=  \left\| \Psi \right\|_q + \left\| \Id \Psi \right\|_q+ \left\| \nabla\Id \Psi\right\|_q,
\end{align*}
where as usual $\nabla\Id \Psi$ will be identified with the Hessian of $\Psi$. We add:

\begin{Definition} Let $q\in (1,\infty)$. We say that $M$ satisfies the \emph{$\mathsf{L}^q$-Calderon-Zygmund inequality}, if there are $C_q>0$, $D_q \geq 0$ such that for all $\Psi\in \mathsf{C}^{\infty}_c(M)$ one has
\begin{align}
\left\| \nabla\Id \Psi\right\|_q\leq C_q \left\| \Delta \Psi\right\|_q+D_q\left\|  \Psi\right\|_q.
\end{align}
\end{Definition}

With this notion at hand, we have:

\begin{Theorem}\label{sobo} Let $q\in (1,\infty)$. If $M$ admits a sequence of Laplacian cut-off functions and satisfies the $\mathsf{L}^q$-Calderon-Zygmund inequality, then $\mathsf{C}^{\infty}_c(M)$ is dense in $\mathsf{H}^{2,q}(M)$.
\end{Theorem} 

\begin{proof} Since $|\Delta \phi(x)|\leq \sqrt{m}|\nabla \Id \phi(x)|_x$ for any $\phi \in \mathsf{C}^{\infty}(M)$, the $\mathsf{L}^q$-Calderon-Zygmund assumption implies that the norm $\left\| \bullet \right\|_{2,q}$ is equivalent to the norm
$$
\left\| \Psi\right\|_{2,q,\rq{}}:=  \left\| \Psi \right\|_q + \left\| \Id \Psi \right\|_q+ \left\| \Delta \Psi\right\|_q
$$
on $\mathsf{C}^{\infty}_c(M)$. Thus, if $(\chi_n)$ is a sequence of Laplacian cut-off functions, and given a $\Psi\in \mathsf{C}^{\infty}(M)$ with $\left\| \Psi\right\|_{2,q,\rq{}}<\infty$, it is sufficient to prove that $\left\|\chi_n\Psi - \Psi \right\|_{2,q,\rq{}}\to 0$. Here,
$$
\int |\chi_n\Psi - \Psi |^q \z \to 0
$$
follows from (C1), (C2) and dominated convergence. Next, using
\begin{align}
\Id (\chi_n\Psi) = \Psi\Id\chi_n+\chi_n \Id \Psi
\end{align}
we get
$$
\int |\Id (\chi_n\Psi) - \Id\Psi |^q \z\leq C  \int | \Psi\Id\chi_n |^q\z +C \int |(\chi_n-1) \Id\Psi |^q \z  \to 0
$$
by (C3), (C1), (C2) and dominated convergence. Furthermore, using (\ref{frf}), we get
\begin{align*}
\int |\Delta(\chi_n\Psi) - \Delta\Psi |^q \z&\leq c_1\int |\chi_n-1|^q |\Delta\Psi|^q\z\\
&+c_2\int|\Psi|^q| \Delta\chi_n|^q\z+c_3\int|\Id \chi_n|^q|\Id \Psi|^q\z\to 0
\end{align*}
by  (C1), (C2), (C3), (C4) and dominated convergence.
\end{proof}

\begin{Remark} 1. If $\mathrm{Ric}\geq -C$ for some $C\geq 0$, then $M$ satisfies the $\mathsf{L}^2$-Calderon-Zygmund inequality, with $C_2=1, D_2=C$, thus in case $q=2$, we recover Theorem 1.1 in \cite{bandara} within the class of of $M$\rq{}s with nonnegative Ricci curvature (noting that the proof of Theorem 1.1 in \cite{bandara} heavily relies on Hilbert space arguments and thus does not extend directly to general $q$\rq{}s). We refer the reader to the monograph \cite{hebe} for results into this direction on the whole $\mathsf{L}^q$-scale.\\
2. Being motivated by the Calderon-Zygmund-Vitali technique \cite{cabre,lihe}, we conjecture:\vspace{2mm}

\emph{ If $M$ is geodesically complete with a nonnegative Ricci curvature, then $M$ satisfies the $\mathsf{L}^q$-Calderon-Zygmund inequality for all $q\in (1,\infty)$.} 
\vspace{2mm}

Corresponding arguments have also been used in \cite{badr} in order to derive interpolation results for $\mathsf{H}^{1,q}(M)$. 

\end{Remark}

\subsection{Positivity preservation and the BMS conjecture}\label{bams}

We continue with:

\begin{Definition}\label{kl} Let $\mathscr{C}\subset \mathsf{L}^1_{\mathrm{loc}}(M)$ be an arbitrary subset. We say that $M$ is $\mathscr{C}$-\emph{positivity preserving}, if the following implication holds true for any $f\in\mathscr{C}$\emph{:}
\begin{align}
 (-\Delta/2+1)f\geq 0\> \text{ as a distribution } \>\>\Rightarrow\>\> f\geq 0.\label{vor}
\end{align}
\end{Definition} 

The factor $1/2$ in (\ref{vor}) is irrelevant in applications (for it can be \lq\lq{}scaled away\rq\rq{} under typical scale invariant assumptions on $M$), but will be convenient for our probabilistic considerations. Definition \ref{kl} is motivated by the work of Braverman, Milatovic and Shubin \cite{Br}, where the authors were interested in essential self-adjointness results for Schrödinger operators on Riemannian manifolds. The precise connection is given in Proposition \ref{stoch} below, which illustrates Definition \ref{kl} very well. Let $E\to M$ denote an arbitrary smooth Hermitian vector bundle with Hermitian covariant derivative $\tilde{\nabla}$, and let $V$ be a Borel section in $\mathrm{End}(E)\to M$. If $V(x):E_x\to E_x$ is self-adjoint for a.e. $x\in M$, we will call $V$ a \emph{potential} on $E$.

\begin{Proposition}\label{stoch}\emph{a)} If $V$ is a potential on $E$ with $V\geq 0$ and $\left|V\right|\in \mathsf{L}^2_{\mathrm{loc}}(M)$, and if $M$ is $\mathsf{L}^2(M)$-positivity preserving, then the operator $\tilde{\nabla}^{\dagger}\tilde{\nabla}/2+V$ (defined on $\Gamma_{\mathsf{C}^{\infty}_{\mathrm{c}}}(M,E)$) is essentially self-adjoint in the Hilbert space $\Gamma_{\mathsf{L}^2}(M,E)$.\\
\emph{b)} Assume $q\in (1,\infty)$, $(V+V^{\dagger })/2\geq 0$, and $\left|V\right|\in \mathsf{L}^q_{\mathrm{loc}}(M)$. Then the operator $\tilde{\nabla}^{\dagger}\tilde{\nabla}/2+V$ (defined on $\Gamma_{\mathsf{C}^{\infty}_{\mathrm{c}}}(M,E)$) is closable in $\Gamma_{\mathsf{L}^q}(M,E)$, and if $M$ is $\mathsf{L}^q(M)$-positivity preserving, then its closure generates a contraction semigroup in the Banach space $\Gamma_{\mathsf{L}^q}(M,E)$.\\
\emph{c)} If $M$ is $\mathsf{C}^{\infty}(M)\cap \mathsf{L}^{\infty}(M)$-positivity preserving, then $M$ is stochastically complete, that is, one has 
$$
\mathbb{P}^x\{t<\zeta\}=1 \>\text{ for some/all $(t,x)\in (0,\infty)\times M$.} 
$$
\end{Proposition}

\begin{proof} a) This statement is certainly implicitely included in \cite{Br}. In fact, it follows immediately from combining the fact that $\tilde{\nabla}^{\dagger}\tilde{\nabla}/2+V$ is essentially self-adjoint if 
$$
\mathrm{Ker}\big((\tilde{\nabla}^{\dagger}\tilde{\nabla}/2+V+1)^*\big)=\{0\},
$$
with a variant of Kato\rq{}s inequality (Theorem 5.7 in \cite{Br}), which states that for any $f\in\Gamma_{\mathsf{L}^1_{\mathrm{loc}}}(M,E)$ with $\tilde{\nabla}^{\dagger}\tilde{\nabla} f\in\Gamma_{\mathsf{L}^1_{\mathrm{loc}}}(M,E)$ one has 
\begin{align}
(-\Delta/2)|f| \leq \mathrm{Re} \big(  (\tilde{\nabla}^{\dagger}\tilde{\nabla}/2) f ,  \mathrm{sign}(f)\big)\>\>\text{ as distributions,}\label{kat}
\end{align}
 where $\mathrm{sign}(f)\in\Gamma_{\mathsf{L}^{\infty}}(M,E)$ is defined by 
$$
\mathrm{sign}(f)(x):=\begin{cases}&\f{f(x)}{|f(x)|}, \>\text{ if } \>f(x)\ne 0\\
&0, \>\text{ else.}\end{cases}
$$
Note here that (\ref{kat}) holds true for any $M$.\\
b) The asserted closability follows from Lemma 2.1 in \cite{mila1} (where the required geodesic completeness is not used), and the other statement can be proved precisely like Theorem 1.3 in \cite{mila1}, noting that the bounded geometry assumption there is only used to deduce that $M$ is $\mathsf{L}^q(M)$-positivity preserving.\\
c) This follows from the fact $M$ is stochastically complete, if and only if any bounded nonnegative solution $f$ on $M$ of $(-\Delta/2+1)f=0$ satisfies $f\equiv 0$ \cite{gri}. 
\end{proof}

\begin{Remark} 1. Proposition \ref{stoch} a) lead Braverman, Milatovic and Shubin to the following conjecture:
\vspace{2mm}

\emph{ If $M$ is geodesically complete, then $M$ is $\mathsf{L}^2(M)$-positivity preserving.} 

\vspace{2mm}

Note that the above \emph{BMS-conjecture} has remained open for more than 10 years by now. On the other hand, it has already been noted in Theorem B.2 in \cite{Br}, that if $M$ admits a sequence of Laplacian cut-off functions, then $M$ is $\mathsf{L}^2(M)$-positivity preserving. Thus, so far, only Riemannian manifolds with a $\mathsf{C}^{\infty}$-bounded geometry have been known to be $\mathsf{L}^2(M)$-positivity preserving. \\
2. Obviously, the test class $\mathsf{C}^{\infty}(M)\cap \mathsf{L}^{\infty}(M)$ in Proposition \ref{stoch} c) can be made much smaller, if necessary. 
\end{Remark}

Using Theorem \ref{main}, we can prove the following variant of the BMS-conjecture, where we regard the Ricci curvature as a potential on $\mathrm{T}^* M$:

\begin{Theorem}\label{bms} Assume that $M$ admits a sequence of Laplacian cut-off functions and that there is a decomposition $\mathrm{Ric}=R_+-R_-$ into potentials $R_{\pm}\geq 0$ on $\mathrm{T}^*M$ such that $\left|R_+\right|\in \mathsf{L}^1_{\mathrm{loc}}(M)$, $\left|R_-\right|\in \mathsf{K}(M)$. Then $M$ is $\mathsf{L}^q$-positivity preserving for any $q\in [1,\infty]$. In particular, the conclusions of Proposition \ref{stoch} hold true.
\end{Theorem}

Note, in particular, that the assumptions in Theorem \ref{bms} on $M$ are satisfied, if $M$ is geodesically complete with a nonnegative Ricci curvature. We will need the generally valid covariant Markoff type result Proposition \ref{markoff} below for the proof (to be precise, for the non-Hilbertian case $q\in [1,\infty]\setminus\{2\}$, which has not been considered at all in \cite{Br}; see \cite{mila2} for a proof that applies to $q\in (1,\infty)$). Define a function
$$
\tilde{c}:[1,\infty]\longrightarrow (0,\infty),\>\tilde{c}(q):=\begin{cases}
&1-\f{1}{q}, \text{ if $1<q<\infty$}\\
&1,\text{ else.}
\end{cases}
$$
Then we have:

\begin{Proposition}\label{markoff} Let $V$ be a potential on the smooth Hermitian vector bundle $E\to M$ which admits a decomposition $V=V_+-V_-$ into potentials $V_{\pm}\geq 0$ such that $\left|V_+\right|\in \mathsf{L}^1_{\mathrm{loc}}(M)$, $\left|V_-\right|\in \mathsf{K}(M)$, and let $\tilde{\nabla}$ be a Hermitian covariant derivative on $E\to M$. If $H_{\tilde{\nabla}}(0)$ denotes the Friedrichs realization in $\Gamma_{\mathsf{L}^2}(M,E)$ of $\tilde{\nabla}^{\dagger}\tilde{\nabla}/2$, then the form sum $H_{\tilde{\nabla}}(V):=H_{\tilde{\nabla}}(0)\dotplus V$ is a well-defined self-adjoint semibounded from below operator in $\Gamma_{\mathsf{L}^2}(M,E)$, and for any $\delta>1$, $q\in [1,\infty]$ there is a $C(\delta,V_-,q)> 0$, which does not depend on $\tilde{\nabla}$, such that for all $t\geq 0$, all $\lambda\in \IC$ with 
$$
\mathrm{Re}(\lambda)< \min\big(C(\delta,V_-,q) ,\min\sigma(H_{\tilde{\nabla}}(V))\big),
$$
and all $k\in \IN$ one has  
\begin{align}
&\left\|   \mathrm{e}^{-t H_{\tilde{\nabla}}(V)}  \mid_{\Gamma_{\mathsf{L}^2\cap \mathsf{L}^q}(M,E)}\right\|_{q}\leq \delta^{\tilde{c}(q)}\mathrm{e}^{tC(\delta,V_-,q)},\label{heat}\\
&\left\|\big(H_{\tilde{\nabla}}(V)-\lambda\big)^{-k}\mid_{\Gamma_{\mathsf{L}^2\cap \mathsf{L}^q}(M,E)}\right\|_{q}\leq \f{ \delta^{\tilde{c}(q)}}{\big(-C(\delta,V_-,q)-\mathrm{Re}(\lambda)\big)^{k}}.\label{res}
\end{align}
\end{Proposition}

\begin{proof} The well-definedness of $H_{\tilde{\nabla}}(V)$ is the main result of \cite{G1}. Let us recall that a refined version of Khasminskii\rq{}s lemma (Lemma 3.9 in \cite{pallara}) implies that
for any $\tilde{w}\in\mathsf{K}(M)$ and any $\delta>1$ there is a $C(\tilde{w},\delta)> 0$ such that
for all $t\geq 0$,
\begin{align}
\sup_{x\in M} \mathbb{E}^x\left[\mathrm{e}^{\int^t_0 \left|\tilde{w}(X_s)\right|\Id
s}1_{\{t<\zeta\}}\right]
\leq \delta \mathrm{e}^{tC(\tilde{w},\delta)}. \label{way0}
\end{align}
If the scalar potentials $w:M\to \IR$, $w_+,w_-:M\to [0,\infty)$ are given by 
\begin{align*}
&w_+(x):=\text{ smallest eigenvalue of }\> V_+(x):E_x\longrightarrow E_x\\
&w_-(x):=\text{ largest eigenvalue of }\> V_-(x):E_x\longrightarrow E_x\\
&w(x):=w_+(x)-w_-(x),
\end{align*}
then $w_+\in\mathsf{L}^1_{\mathrm{loc}}(M)$, $w_-\in\mathsf{K}(M)$. Let $f\in \Gamma_{\mathsf{L}^2\cap \mathsf{L}^q}(M,E)$. \\
Let us first prove (\ref{heat}). One has the semigroup domination \cite{guen}
\begin{align}
\left\|\mathrm{e}^{-t H_{\tilde{\nabla}}(V)}f\right\|_{q}\leq \left\|\mathrm{e}^{-t H_{\Id}(w)}|f|\right\|_{q},\label{sem}
\end{align}
noting that $H_{\Id}(w)$ is a usual scalar Schrödinger operator of the form $-\Delta/2+w$ in $\mathsf{L}^2(M)$, since we consider the differential \lq\lq{}$\Id$\rq\rq{} as a covariant derivative on the trivial line bundle over $M$. Assume first that $1<q<\infty$ and let $q^*:=1/(1-1/q)$ be the Hölder exponent of $q$. Then by (\ref{sem}), the Feynman-Kac formula 
$$
\mathrm{e}^{-t H_{\Id}(w)}h(x)= \mathbb{E}^x\left[1_{\{t<\zeta\}}\mathrm{e}^{-\int^t_0 w(X_s)\Id s}h(X_s)\right],\>h\in\mathsf{L}^2(M),
$$
and Hölder\rq{}s inequality we get
\begin{align*}
\left\|\mathrm{e}^{-t H_{\tilde{\nabla}}(V)}f\right\|_q^q&\leq \int_M \mathbb{E}^x\left[1_{\{t<\zeta\}}\mathrm{e}^{\int^t_0 w_-(X_s)\Id s}|f|(X_s)\right]^q \mu(\Id x)\\
&\leq \int_M \mathbb{E}^x\left[1_{\{t<\zeta\}}\mathrm{e}^{q^*\int^t_0 w_-(X_s)\Id s}\right]^{q/q^*} \mathbb{E}^x\left[1_{\{t<\zeta\}}|f|^q(X_s)\right] \mu(\Id x),
\end{align*}
which using (\ref{way0}) and $\int_M p(t,x,y) \mu(\Id x)\leq 1$ is
$$
\leq \delta^{q/q^*}\mathrm{e}^{\f{tC(\delta,qw_-)}{q^*}} \left\|f\right\|_q^q\equiv \delta^{q/q^*}\mathrm{e}^{tC(\delta,V_-,q)q} \left\|f\right\|_q^q.
$$
If $q=\infty$, then the desired bound follows immediately from (\ref{sem}), the Feynman-Kac formula for $\mathrm{e}^{-t H_{\Id}(w)}$ and (\ref{way0}), for example with
$$
C(\delta,V_-,\infty):=C(\delta,w_-).
$$
If $q=1$, then we can proceed as follows: Let $\bigcup_n K_n =M$ be a relatively compact exhaustion of $M$. Then we have 
\begin{align*}
&\int \mathrm{e}^{-t H_{\Id}(w)} |f| \cdot 1_{K_n} \z=\int  |f| \mathrm{e}^{-t H_{\Id}(w)}  1_{K_n}\z\\
&\leq \left\|\mathrm{e}^{-t H_{\Id}(w)}\mid_{\mathsf{L}^2(M)\cap \mathsf{L}^{\infty}(M) }\right\|_{\infty}  \left\|f\right\|_{1},
\end{align*}
where we have used the self-adjointness of $\mathrm{e}^{-t H_{\Id}(w)}$ for the equality, and the $q=\infty$ case for the inequality. Using monotone convergence this implies
$$
\left\|\mathrm{e}^{-t H_{\Id}(w)}|f|\right\|_{1}\leq  \delta\mathrm{e}^{tC(\delta,V_-,\infty)}\left\|f\right\|_{1}
$$
and the claim follows from (\ref{sem}).\\
Finally, in order to see (\ref{res}), one now just has to note that by the above, for any complex number $\lambda$ the integral 
$$
\f{1}{(k-1)!}\int^{\infty}_0 t^{k-1}\mathrm{e}^{\lambda t} \mathrm{e}^{- t H_{\nabla}(V)} f\Id t
$$
converges in $\Gamma_{\mathsf{L}^q}(M,E)$ if $\mathrm{Re}(\lambda)< C(\delta,V_-,q)$, and in $\Gamma_{\mathsf{L}^2}(M,E)$ if $\mathrm{Re}(\lambda)< \min\sigma(H_{\tilde{\nabla}}(V))$, the latter being equal to $\big(H_{\tilde{\nabla}}(V)-\lambda\big)^{-k}f$.
\end{proof}

Note that we do not have to assume anything on $M$ in Proposition \ref{markoff}, which relies on the fact that we consider $\mathsf{L}^q\leadsto\mathsf{L}^q$ smoothing. Of course one needs bounds of the form $p(t,x,y)\leq C(t)$, if one is interested in $\mathsf{L}^q\leadsto\mathsf{L}^{\tilde{q}}$, $1\leq q\leq \tilde{q}\leq \infty$ smoothing results (cf. \cite{guen}). 

\begin{proof}[Proof of Theorem \ref{bms}] We will use the notation from Proposition \ref{markoff} on covariant Schrödinger operators. Let $q\in [1,\infty]$. We are going to prove that for all $f\in \mathsf{L}^q(M)$ with
$$
(-\Delta/2+1)f\geq 0
$$
and all $0\leq \phi\in\mathsf{C}^{\infty}_{\mathrm{c}}(M)$, one has $\int f \phi \z\geq 0$. To this end, let $H:=H_{\Id}(0)\geq 0$ denote the Friedrichs realization of $-\Delta/2$ in $\mathsf{L}^2(M)$ and let $H_1:=H_{\nabla}(\mathrm{Ric})\geq 0$ denote the Friedrichs realization in $\Omega^1_{\mathsf{L^2}}(M)$ of the Laplace-Beltrami operator $-\Delta_1 /2= \nabla^{\dagger}\nabla/2+\mathrm{Ric}/2$ on $1$-forms, where the latter equality is precisely Weitzenböck\rq{}s formula. Note that by geodesic completeness, $H$ and $H_1$ are the respectively unique self-adjoint realizations. Let
\begin{align*}
\Psi:=&(H+1)^{-1}\phi=\int^{\infty}_0 \mathrm{e}^{-t} \mathrm{e}^{-t H} \phi \ \Id t\\
&=\int^{\infty}_0\mathrm{e}^{-t }\int_M p(t,\bullet,y)\phi(y)\mu(\Id y)\Id t:M\longrightarrow  [0,\infty).
\end{align*}
Then $\Psi$ is smooth with $(-\Delta/2+1)\Psi=\phi$ \cite{gri}, and by Theorem \ref{markoff} we have
\begin{align*}
(H+1)^{-1}: \mathsf{L}^2(M)\cap \mathsf{L}^{\tilde{q}}(M)\longrightarrow  \mathsf{L}^2(M)\cap \mathsf{L}^{\tilde{q}}(M)\>\text{ for all $\tilde{q}\in [1,\infty]$},
\end{align*}
thus $\Psi,\Delta\Psi\in\mathsf{L}^{\tilde{q}}(M)$. We will also show later that
\begin{align}
|\Id\Psi|\in\mathsf{L}^{\tilde{q}}(M)\>\text{ for all $\tilde{q}\in [1,\infty]$},\label{zuz}
\end{align}
which we assume for the moment. Up to (\ref{zuz}), the proof of Theorem \ref{bms} will be complete, if we can show 
\begin{align}
\int f (-\Delta/2+1)\Psi\z=\lim_{n\to\infty}\int f (-\Delta/2+1)(\chi_n\Psi) \z,\label{zuzi}
\end{align}
with $(\chi_n)$ a sequence of Laplacian cut-off functions. In order to see (\ref{zuzi}), note that it is clear from (C1), (C2) and dominated convergence that
$$
\int f \Psi \z=\lim_{n\to\infty}\int f \chi_n\Psi \z.
$$
Furthermore, using (\ref{frf}) we get
\begin{align*}
\int f \Delta(\chi_n\Psi) \z=&\int f \Psi\Delta\chi_n \z+2\int (f\Id \chi_n,\Id\Psi)\z+\int_M f \chi_n\Delta\Psi \z.
\end{align*}
Here, 
\begin{align}
\lim_{n\to\infty}\int f \Psi\Delta\chi_n \z=0\>\> \text{ and  }\>\lim_{n\to\infty}\int f \chi_n\Delta\Psi \z=\int f \Delta\Psi \z\nn
\end{align}
follow, respectively, from (C4), and (C1), (C2) combined with dominated convergence, and finally 
$$
\lim_{n\to\infty}\int (f\Id \chi_n,\Id\Psi)\z=0
$$
follows from (C3) and (\ref{zuz}), and we are done.

It remains to prove (\ref{zuz}): The geodesic completeness of $M$ implies (see for example the appendix of \cite{anton})
$$
\Id \mathrm{e}^{-t H}\phi= \mathrm{e}^{-t H_1}\Id\phi\>\text{ for all $t>0$,}
$$
 so that using
$$
\int^{\infty}_0 \mathrm{e}^{-t} \mathrm{e}^{-t H} \Id t= (H+1)^{-1},\>\int^{\infty}_0 \mathrm{e}^{-t} \mathrm{e}^{-t H_1} \Id t= (H_1+1)^{-1}
$$
one easily gets the identity
$$
\Id\Psi=(H_1+1)^{-1}\Id \phi.
$$
Thus (\ref{zuz}) is implied by Theorem \ref{markoff}, since clearly $\Id\phi\in \Omega^1_{ \mathsf{L}^{\tilde{q}}  }(M)$ for all $\tilde{q}\in [1,\infty]$. This completes the proof of Theorem \ref{bms}.

\end{proof}

Let us add some comments on Theorem \ref{bms}: The property 
$$
\text{\lq\lq{}$\mathscr{C}$-positivity preservation\rq\rq{} }
$$
and thus also the BMS-conjecture both admit an equivalent formulation which makes sense on an arbitrary local Dirichlet space \cite{sturm,lenz}. On the other hand, the concept of sequences of cut-off functions does not immediately extend to this setting, so that it would be very interesting to find a proof of the fact that geodesic completeness and a nonnegative Ricci curvature imply $\mathsf{L}^q$-positivity preservation for any $q\in [1,\infty]$ which does not use sequences of cut-off functions.\\
Finally, we would like to remark that Proposition \ref{stoch}.c) and Theorem \ref{bms} combine to a completely new proof of the classical \cite{yau} fact that geodesic completeness and a nonnegative Ricci curvature imply stochastic completeness. Note, however, that in the context of stochastic completeness, one can in fact allow certain unbounded negative parts of the Ricci curvature, namely, it is enough \cite{hsu1} to assume that the Ricci curvature is bounded from below in radial direction by a quadratic function of $\Id(\bullet,\mathscr{O})$ (see also \cite{hack, stroock} for textbook versions and variations of the latter result).

\end{document}